\documentclass{amsart}
\usepackage{amssymb,latexsym}
\theoremstyle{plain}
\newtheorem{theorem}{Theorem}

\newtheorem{proposition}{Proposition}
\newtheorem{lemma}{Lemma}
\theoremstyle{definition}
\newtheorem{definition}{Definition}
\newtheorem{example}{Example}
\newtheorem{notation}{Notation}
\newtheorem{remark}{Remark}
\newtheorem{notationr}{Notation and Remark}

\date{}

\begin{document}

\title[Curvilinear subschemes of Veronese varieties]
{A partial stratification of secant varieties of Veronese varieties via curvilinear subschemes}
\author{Edoardo Ballico, Alessandra Bernardi}
\address{Dept. of Mathematics\\
  University of Trento\\
38123 Povo (TN), Italy}
\address{GALAAD, INRIA M\'editerran\'ee,  BP 93, 06902 Sophia 
Antipolis, France.}
\email{ballico@science.unitn.it, alessandra.bernardi@inria.fr}
\thanks{The authors were partially supported by CIRM of FBK Trento 
(Italy), Project Galaad of INRIA Sophia Antipolis M\'editerran\'ee 
(France), Institut Mittag-Leffler (Sweden), Marie Curie: Promoting science (FP7-PEOPLE-2009-IEF), MIUR and GNSAGA of 
INdAM (Italy).}

\begin{abstract}
We give a partial ``~quasi-stratification~'' of the secant varieties of the order $d$ Veronese variety $X_{m,d}$ of $\mathbb {P}^m$. It
covers the set $\sigma _t(X_{m,d})^{\dagger}$ of all points lying on the linear span of curvilinear subschemes of $X_{m,d}$, but
two ``~quasi-strata~'' may overlap. For low border rank two different ``~quasi-strata~'' are disjoint and we compute the
symmetric rank of their elements. Our tool is the Hilbert schemes of curvilinear subschemes of Veronese varieties. To get a stratification
we attach to each $P\in \sigma _t(X_{m,d})^{\dagger}$ the minimal label of a quasi-stratum containing it.
\end{abstract}

\maketitle

\section*{Introduction}\label{S1}
Let $\nu _d: \mathbb {P}^m \hookrightarrow \mathbb {P}^{\binom{m+d}{m}-1}$ be the order $d$ Veronese embedding with $d\ge 3$. We write $X_{m,d}:= \nu _d(\mathbb {P}^m)$.
An element of $X_{m,d}$ can be described both as the projective class of a $d$-th power of a homogeneous linear form in $m+1$ variables and as the projective class of a completely decomposable symmetric $d$-modes tensor.  
In many applications like Chemometrics (see e.g. \cite{tb}), Signal Processing (see e.g. \cite{lc}), Data Analysis (see e.g. \cite{ay}), Neuroimaging (see e.g. \cite{dcdsrvgv}), Biology (see e.g. \cite{mwlwcglv}) and many others, the knowledge of the minimal decomposition of a  tensor in terms of completely decomposable tensors turns out to be extremely useful. This kind of decomposition is strictly related with the concept of secant varieties of  varieties parameterizing tensors (if the tensor is symmetric one has to deal with secant varieties of Veronese varieties).

Let $Y \subseteq \mathbb{P}^N$ be an integral and non-degenerate variety defined over an algebraically closed field $\mathbb {K}$ of characteristic zero.\\
For any point $P\in \mathbb{P}^N$ the \emph{$Y$-rank $r_Y(P)$ of $P$} is the minimal cardinality of a finite set of points $S\subset Y$ such that $P\in \langle S\rangle$, where
$\langle \ \ \rangle$ denote the linear span:
\begin{equation}\label{rank}
 r_Y(P):=\min \{s\in \mathbb{N} \; | \; \exists \,  S\subset Y, \, \sharp(S)=s,  \hbox{ with } \, P\in \langle S \rangle \}.
\end{equation}

If $Y$ is the Veronese variety $X_{m,d}$
the $Y$-rank is also called  the ``~symmetric tensor rank~''.  The minimal set of points $S\subset X_{m,d}$ that realizes the symmetric tensor rank of  a point $P\in X_{m,d}$ is also said the set that realizes either the ``~CANDECOMP/PARAFAC decomposition~''  or  the ``~canonical decomposition~'' of $P$. 

Set $X:= X_{m,d}$. The natural geometric object that one has to study in order to compute the symmetric tensor rank either of a symmetric tensor or of a homogeneous polynomial is the set that parameterizes points in $\mathbb{P}^N$ having $X$-rank smaller or equal than a fixed value $t\in \mathbb{N}$. For each integer $t\ge 1$ let the $t$-th secant variety  $\sigma _t(X)\subseteq \mathbb{P}^N $ of a variety $X\subset \mathbb{P}^N$ be the Zariski closure in $\mathbb{P}^N$ of the union of all $(t-1)$-dimensional
linear subspaces spanned by $t$ points of $X\subset \mathbb{P}^N$:
\begin{equation}\label{sigma}
\sigma _t(X):= \overline{\bigcup_{P_1, \ldots , P_t\in X}\langle P_1, \ldots , P_t\rangle}
\end{equation} 

For each $P\in \mathbb{P}^N$ the \emph{border rank $b_X(P)$ of $P$}
is the minimal integer $t$ such that $P\in \sigma _t(X)$:
\begin{equation}\label{borderrank}
 b_X(P):=\min\{t\in \mathbb{N} \; | \; P\in \sigma _t(X)\}.
\end{equation}
 
We indicate with $\sigma_t^0(X)$ the set of the elements belonging to $\sigma_t(X)$ of fixed $X$-rank $t$:
\begin{equation}\label{sigma0}
\sigma ^0_t(X):= \{P\in \sigma _t(X) \; | \; r_X(P)=t\}
\end{equation} 
Observe that if $\sigma _{t-1}(X)\ne \mathbb{P}^N$, then
$\sigma ^0_t(X)$ contains a non-empty open subset of $\sigma _t(X)$.

Some of the recent papers on algorithms that are able to compute the symmetric tensor rank of a symmetric tensor (see \cite{bgi}, \cite{bb1}, \cite{bcmt}) use the idea of giving a stratification of the $t$-th secant variety of the Veronese variety via the symmetric tensor rank. In fact, since $\sigma_t(X)=\overline{\sigma_t^0(X)}$, the elements belonging to $\sigma_t(X)\setminus (\sigma_t^0(X) \cup \sigma_{t-1}(X))$ have $X$-rank strictly bigger than $t$. What some of the known algorithms for computing the symmetric rank of a symmetric tensor $T$ do is firstly to test the equations of the secant varieties of the Veronese varieties (when known) in order to find the $X$-border rank of $T$, and secondly to use (when available) a stratification via the symmetric tensor rank of $\sigma_t(X)$.
For the state of the art on the computation of the symmetric rank of a symmetric tensor see  \cite{cs}, \cite{bcmt}, \cite{lt}
Theorem 5.1, \cite{bgi}, \S 3, for the case of rational normal curves, \cite{bgi} for the case $t=2,3$, \cite{bb1} for $t=4$.

Moreover, the recent paper \cite{BucBuc}, has shown the importance of the study of the smoothable 0-dimensional schemes in order to understand the structure of the points belonging to secant varieties to Veronese varieties.

We propose here the computation of the symmetric tensor rank of a particular class of the symmetric tensors whose symmetric border rank is strictly less than its symmetric rank. We will focus on those symmetric tensors that belong to the linear span of a reduced $0$-dimensional curvilinear sub-scheme of the Veronese variety. We will indicate in Notation \ref{dag} this set as $\sigma _t(X)^{\dagger}$. We use a well-known stratification of the subset of the Hilbert scheme $\mbox{Hilb}^t(\mathbb {P}^m)_c$ of curvilinear zero-dimensional subschemes
of $\mathbb {P}^m$ with degree $t$. Taking the unions of all $\langle \nu _d(A)\rangle$, $A\in \mbox{Hilb}(\mathbb {P}^m)_c$, we get $\sigma _t(X)^{\dagger}$. From each stratum $U$ of  $\mbox{Hilb}^t(\mathbb {P}^m)_c$ we get a quasi-stratum $\cup _{A\in U}\langle A\rangle$
of $\sigma _t(X)$. In this way we do not obtain a stratification of $\sigma _t(X)^{\dagger}$, because a point of $\sigma _t(X)^{\dagger}$ may be in the intersection of the linear spans of elements of two different strata of $\mbox{Hilb}^t(\mathbb {P}^m)_c$. We may get a true
stratification of $\sigma _t(X)^{\dagger}$ taking a total ordering of the set of all strata of $\mbox{Hilb}^t(\mathbb {P}^m)_c$ and assigning to any $P\in \sigma _t(X)^{\dagger}$ only the stratum of $\mbox{Hilb}^t(\mathbb {P}^m)_c$ with minimal label among the strata
with $P$ in their image. The strata of $\mbox{Hilb}^t(\mathbb {P}^m)_c$ have a natural partial ordering with maximal element $(1,\dots ,1)$
corresponding to $\sigma ^0_t(X)$ and the next maximal one $(2,\dots ,1)$ (Notation \ref{nn1} and Lemma \ref{zz1}). Hence $\sigma _t(X)^{\dagger} \setminus \sigma _t^0(X)$ has a unique maximal quasi-stratum and we may speak about the general element of the unique component of maximal dimension
of $\sigma _t(X)^{\dagger} \setminus \sigma _t^0(X)$. If $t\ \le (d+1)/2$, then our quasi-stratification of $\sigma _t(X)^{\dagger}$ is a true stratification, because the images of two different strata of $\mbox{Hilb}(\mathbb {P}^m)_c$ are disjoint.
We may give the lexicographic ordering to the labels of $\mbox{Hilb}^t(\mathbb {P}^m)_c$ to get a total ordering and hence a true stratification of $\sigma _t(X)^{\dagger}$, but it is rather artificial: there is no reason to say that the quasi-stratum
$(3,1,\dots ,1)$ comes before the quasi-stratum $(2,2,1,\dots ,1)$.

For very low $t$ (i.e. $t\le \lfloor (d-1)/2\rfloor $), we will describe the structure of $\sigma_t(X)^{\dagger}$: we will give its dimension, its codimension in $\sigma_t(X)$ and the dimension of each stratum
(see Theorem \ref{e1}). Moreover in the same theorem we will show that for such values of $t$, the symmetric border rank of the projective class of a homogeneous polynomial $[F]\in \sigma_t(X)\setminus (\sigma_t^0(X)\cup \sigma_{t-1}(X))$ is computed by a unique $0$-dimensional subscheme $W_F\subset X$ and that the generic $[F]\in \sigma_t(X)^{\dagger}$ is of the form $F=L^{d-1}M+L_1^d+\cdots + L_{t-2}^d$ with $L,L_1, \ldots , L_{t-2},M$ linear forms.
To compute the dimension of the $3$ largest
strata of our stratification we will use Terracini's lemma (see Propositions \ref{j1}, \ref{j2} and \ref{j3}).

We will also prove several results on the symmetric ranks of points $P\in \mathbb{P}^N$ whose border
rank is computed by a scheme related to our stratification
(see Proposition \ref{e2+} and Theorem \ref{f2}). In all cases that we will be able to compute,
we will have  $b_X(P)+r_X(P) \le 3d-2$,
but we will need also additional conditions on the scheme computing $b_X(P)$ when $b_X(P)+r_X(P)\ge
2d+2$. 

\section{The quasi-stratification}\label{S2}

For any scheme $T$ let $T_{red}$ denote its reduction. We begin this section by recalling the well known stratification of the curvilinear $0$-dimensional subschemes of any smooth connected projective variety $Y\subset \mathbb{P}^r$.(\footnote{Expert readers can skip this section and refer to it only for Notation.})

\begin{notation}\label{n1}
For any integral projective variety $Y\subset \mathbb {P}^r$ let $\beta (Y)$ be the maximal positive integer such that every $0$-dimensional scheme $Z\subset Y$
with $\deg (Z)\le \beta (Y)$ is linearly independent, i.e. $\dim (\langle Z\rangle ) = \deg (Z)-1$ (see \cite{bgl}, Lemma 2.1.5, or \cite{bb1}, Remark 1, for the Veronese varieties).
\end{notation}

\begin{remark}\label{ee1}
Let $Z \subset \mathbb {P}^m$ be any $0$-dimensional scheme. If $\deg (Z) \le d+1$, then $h^1(\mathbb {P}^m,\mathcal {I}_Z(d))=0$. If $Z$ is the union of $d+2$ collinear points, then
$h^1(\mathbb {P}^m,\mathcal {I}_Z(d))=1$. Therefore $\beta (X_{m,d}) =d+1$.
\end{remark}

\begin{notation}\label{Az}
Fix an integer $t\ge 1$. 
Let 
$A(t)$ be the set of all non-increasing sequences $t_1\ge t_2 \ge \cdots \ge t_s \geq 0$ such that
$\sum _{i=1}^{s} t_i = t$.
\\
For each such sequence $\underline{t} = (t_1,\dots ,t_s)$ let
$l(\underline{t})$ be the number of the non zero $t_i$'s, for $i=1, \ldots , s$.
\\
Set $B(t):= A(t) \setminus \{(1,\dots ,1)\}$ in which
the string $(1,\dots ,1)$ has $t$ entries. 
\end{notation}

$A(t)$ is the set of all partitions of the integer $t$. The integer $l(\underline{t})$ is the length of the partition $\underline{t}$.

\begin{definition}
Let $Y\subset \mathbb{P}^r$ be a smooth and connected projective variety
of dimension $m$. For every positive integer $t$ let $\mbox{Hilb}^t(Y)$ denote the Hilbert scheme of all degree $t$ $0$-dimensional subschemes
of $Y$.
\end{definition}

If $m\le 2$, then $\mbox{Hilb}^t(Y)$ is smooth and irreducible (\cite{f}, Propositions 2.3 and 2.4, \cite{g}, page 4).

We now introduce some subsets of $\mbox{Hilb}^t(Y)$ that will give the claimed stratification.

\begin{notationr}\label{dots} Let $Y\subset \mathbb{P}^r$ be a smooth connected projective variety of dimension $m$.

\begin{itemize}
\item For every positive integer $t$ let $\mbox{Hilb}^t(Y)_0$ be the set of all disjoint unions of $t$ distinct points of $Y$. 

Observe that $\mbox{Hilb}^t(Y)_0$ is a smooth
and irreducible quasi-projective variety of dimension $mt$. If $m\le 2$, then $\mbox{Hilb}^t(Y)_0$ is dense
in $\mbox{Hilb}^t(Y)$ (see \cite{f}, \cite{g}, page 4). For arbitrary $m=\dim(Y)$ the irreducible scheme $\mbox{Hilb}^t(Y)_0$ is always open in $\mbox{Hilb}^t(Y)$.
\item  Let
$\mbox{Hilb}^t(Y)_+$ be the closure of $\mbox{Hilb}^t(Y)_0$ in the reduction $\mbox{Hilb}^t(Y)_{red}$ of the scheme $\mbox{Hilb}^t(Y)$. The elements of $\mbox{Hilb}^t(Y)_+$ are called
the {\it smoothable} degree $t$ subschemes of $Y$. 
\\
If $t\gg m \ge 3$, then there are non-smoothable degree $t$ subschemes of $Y$ (\cite{i}, \cite{g}, page 6).
\item An element $Z\in \mbox{Hilb}^t(Y)$ is called {\it curvilinear} 
if at each point $P\in Z_{red}$ the Zariski tangent space of $Z$ has dimension $\le 1$ (equivalently, $Z$ is contained in a smooth subcurve of $Y$). Let $\mbox{Hilb}^t(Y)_c$ denote the set of all degree $t$ {\it curvilinear} subschemes
of $Y$. $\mbox{Hilb}^t(Y)_c$ is a smooth open subscheme  of $\mbox{Hilb}^t(Y)_+$ (\cite{r}, bottom of page 86). It contains $\mbox{Hilb}^t(Y)_0$. 
\end{itemize}
\end{notationr}
 
 Fix now $O\in Y$ with $Y\subset \mathbb{P}^r$ being a smooth connected projective variety of dimension $m$. Following \cite{g}, page 3, we state the corresponding result for the punctual Hilbert scheme of $\mathcal {O}_{Y,O}$,
 i.e. the scheme parametrizing all degree $t$ zero-dimensional schemes $Z \subset Y$ such that $Z_{red} = \{O\}$
 (here instead of ``~curvilinear~'' several references use the word ``~collinear~'') .
 
\begin{remark}\label{pallino}
For each integer
$t>0$ the subset of the punctual Hilbert scheme parametrizing the degree $t$ curvilinear subschemes of $Y$ with $P$ as its reduction is smooth, connected and of dimension $(t-1)(m-1)$.
\end{remark}

\begin{notation}\label{strati}
Fix an integer $s>0$ and a non-increasing sequence of integers $t_1 \ge \dots \ge t_s>0$ such that $t_1+\cdots +t_s=t$ and $\underline{t} = (t_1,\dots ,t_s)$. Let $\mbox{Hilb}^t(Y)_c[t_1,\dots ,t_s]$ denote the subset
of $\mbox{Hilb}^t(Y)_c$ parametrizing all elements of $\mbox{Hilb}^t(Y)_c$ with $s$ connected components of degree $t_1,\dots ,t_s$ respectively. We also write it as $\mbox{Hilb}^t(Y)_c[\underline{t}]$. 
\end{notation}

\begin{remark}\label{biggest}
Since the support of each component $\mbox{Hilb}^t(Y)_c[\underline{t}]$ varies in the $m$-dimensional variety $Y\subset \mathbb{P}^r$, the theorem on the punctual Hilbert scheme quoted in Remark \ref{pallino} says that
$\mbox{Hilb}^t(Y)_c[t_1,\dots ,t_s]$ is an irreducible algebraic set of dimension $ms + \sum _{i=1}^{s} (t_i-1)(m-1) = 
mt+s-t$, i.e. of codimension $t-s$
in $\mbox{Hilb}^t(Y)_c$.
Each stratum $\mbox{Hilb}^t(Y)_c[\underline{t}]$ is non-empty, irreducible and different elements
of $A(t)$ give disjoint strata, because any curvilinear subscheme has a unique type $\underline{t}$. 

Hence if
$t\ge 2$ we have:
$$\mbox{Hilb}^t(Y)_c =
\sqcup _{\underline{t}\in A(t)} \mbox{Hilb}^t(Y)_c[\underline{t}] =\mbox{Hilb}^t(Y)_0 \bigsqcup  \sqcup _{\underline{t}\in
B(t)} \mbox{Hilb}^t(Y)_c[\underline{t}].$$

Different strata may have the same codimension, but there is a unique stratum of codimension
$1$: it is the stratum with label $(2,1,\dots ,1)$. This stratum parametrizes the disjoint unions of a tangent vector to $Y$ and $t-2$ disjoint points of $Y$.
\end{remark}

\begin{notation}\label{nn1}
Take now a partial ordering $\preceq$ on $A(t)$ writing $(a_1,\dots ,a_x) \preceq (b_1,\dots ,b_y)$ if and only
if $\sum _{j=1}^{i} a_j \le \sum _{j=1}^{i} b_j$ for all integers $i\ge 1$, in which we use the convention $a_j:=0$ for all $j>x$ and $b_j=0$ for all $j>y$. In the theory of partitions the partial ordering
$\preceq$ is called the {\emph {dominance partial ordering}}.
\end{notation}

The next lemma is certainly well-known, but we were unable to find a reference.

\begin{lemma}\label{zz1}
Fix $ (t_1,\dots ,t_s) \in B(t)$.

\quad (a) The stratum $\mbox{Hilb}^t(Y)_c[t_1,\dots ,t_s]$ is in the closure of the stratum \\
$\mbox{Hilb}^t(Y)_c[2,1,\dots ,1]$.

\quad (b) If $t_1\ge 3$, then the stratum $\mbox{Hilb}^t(Y)_c[t_1,\dots ,t_s]$ is in the closure of the stratum $\mbox{Hilb}^t(Y)_c[3,1,\dots ,1]$.

\quad (c) if $t_2\ge 2$, then the stratum $\mbox{Hilb}^t(Y)_c[t_1,\dots ,t_s]$ is in the closure of the stratum $\mbox{Hilb}^t(Y)_c[2,2,1,\dots ,1]$.
\end{lemma}

\begin{proof}
We only check part (c), because the proofs of parts (a) and (b) are similar. Fix $Z\in \mbox{Hilb}^t(Y)_c[t_1,\dots ,t_s]$. Take a smooth curve $C\subseteq Y$ containing
$Z$ and write $Z = \sum _{i=1}^{s} t_i P_i$ with $P_i\ne P_j$ for all $i\ne j$. Since $t_1\ge 2$ the effective divisor $t_1P_1$ is a flat degeneration of a family of divisors $\{Z_\lambda\}$
of $C$ in which each $Z_\lambda$ is the disjoint union of a connected degree $2$ divisor and $t_1-2$ distinct points. Similarly, the divisor $t_2P_2$ is a flat degeneration of a family of divisors $\{Z'_\lambda\}$
of $C$ in which each $Z'_\lambda$ is the disjoint union of a connected degree $2$ divisor and $t_2-2$ distinct points. Obviously for each $i\ge 3$ the divisor $t_iP_i$ is smoothable inside $C$, i.e. it is a flat
degeneration of flat family of $t_i$ distinct points. The product of these parameter spaces is a parameter space for a deformation of $Z$ to a flat family of elements of $\mbox{Hilb}^t(C)_c[2,2,1,\dots ,1]$.
Since $C\subseteq Y$, we have $\mbox{Hilb}^t(C)_c[2,2,1,\dots ,1]\subseteq \mbox{Hilb}^t(Y)_c[2,2,1,\dots ,1]$ and the inclusion is a morphism. Hence (c) is true.
\end{proof}

We recall the following lemma (\cite{bgl}, Lemma 2.1.5, \cite{bgi}, Proposition 11, \cite{bb}, Remark 1).

\begin{lemma}\label{a1}
Let $Y\subset \mathbb {P}^r$ be a smooth and connected subvariety. Fix an integer $k$ such that $k \le \beta (Y)$, where $\beta (Y)$ is defined in Notation  \ref{n1}, and $P\in \mathbb {P}^r$.
Then $P\in \sigma _k(Y)$ if and only if there exists a smoothable $0$-dimensional scheme
$Z \subset Y$ such that $\deg (Z) = k$ and $P\in \langle Z\rangle$. 
\end{lemma}

The following lemma shows a very special property of the curvilinear subschemes.

\begin{lemma}\label{a2}
Let $Y\subset \mathbb {P}^r$ be a smooth and connected subvariety. Let $W\subset Y$ be a linearly independent curvilinear subscheme of $Y$.
Fix a general $P\in \langle W\rangle$. Then $P\notin \langle W'\rangle$ for any $W'\subsetneq W$.
\end{lemma}

\begin{proof}
A curvilinear subscheme of a smooth variety is locally a complete intersection. Hence it is Gorenstein. Hence the lemma is a particular case of \cite{bgl}, Lemma 2.4.4. It may be also proved in the following
elementary way, which in addition gives a description of $\langle W\rangle \setminus (\cup _{W'\subsetneq W}\langle W'\rangle )$. Fix any $W'\subsetneq W$. Since $\deg (W') < \deg (W) \le \beta (Y)$, we have $\dim (\langle W'\rangle )-1
< \dim (\langle W\rangle )$. Hence it is sufficient to show that $W$ has only finitely many proper subschemes. Take a smooth quasi-projective curve $C\supset W$. $W$ is an effective Cartier divisor $\sum _{i=1}^{s} b_iP_i$ with
$P_i\in C$, $b_i>0$ for all $i$ and $\sum _{i=1}^{s} b_i =\deg (W)$. Any $W'\subseteq W$ is of the form $\sum _{i=1}^{s} a_iP_i$ for some integers $a_i$ such that $0 \le a_i \le b_i$ for all $i$.
\end{proof}

We introduce the following Notation.

\begin{notation}
For each integral variety $Y\subset \mathbb{P}^r$
and each $Q\in Y_{reg}$ let $[2Q,Y]$ denote the first infinitesimal neighborhood of $Q$ in $Y$, i.e. the closed subscheme of $Y$ with $(\mathcal {I}_{Q,Y})^2$ as its ideal sheaf. We call
any $[2Q,Y]$, with $Q\in Y_{reg}$, a double point of $Y$.
\end{notation}

\begin{remark} Observe that  $[2Q,Y]_{red} = \{Q\}$ and $\deg ([2Q,Y]) = \dim (Y)+1$.
\end{remark}

The following observation shows that Lemma \ref{a2} fails for some non-curvilinear subscheme.

\begin{remark}\label{a3}
Assume that $Y \subset \mathbb {P}^r$ is smooth and of dimension $\ge 2$. Fix a smooth subvariety $N\subseteq Y$ such that $\dim (N)=2$ and any $Q\in N$.
Since $N$ is embedded in
$\mathbb {P}^r$, the linear space $\langle [2Q,Y]\rangle$ is a $2$-dimensional space (it is usually called the Zariski tangent space or embedded Zariski tangent space of $Y$ at $Q$). Fix any $P\in \langle [2Q,Y]\rangle$.
If $P=Q$, then $P\in \langle \{Q\}\rangle$. If $P\neq Q$, then the plane $\langle [2Q,Y]\rangle$ intersects
$[2Q,Y]$ in a degree $2$ subscheme $[2Q,Y]_P$ and $P\in \langle [2Q,Y]_P\rangle$ that is a line. 
\end{remark}

\begin{notation}\label{dag}
For any integer $t >0$ let $\sigma _t(X)^{\dagger}$ denote the set of all $P\in \sigma _t(X)\setminus (\sigma ^0_t(X)\cup \sigma_{t-1}(X))$ such that there
is a curvilinear degree $t$ subscheme $Z\subset X_{reg}$ such that $P\in \langle Z\rangle$.
\end{notation} 

\begin{remark}\label{dd1} Let $X\subset \mathbb{P}^N$ be the Veronese variety $X_{m,d}$ with $N={n+d \choose d}-1$.
Take $P\in \sigma _t(X)^{\dagger}$ and a curvilinear degree $t$ subscheme $Z\subset X_{reg}$ such that $P\in \langle Z\rangle$. The curvilinear scheme $Z$ has a certain number, $s$, of connected components
of degrees $t_1,\dots ,t_s$ respectively with $t_1\ge \cdots \ge t_s$, but we cannot associate the string $(t_1,\dots ,t_s)$ to $P$, because $Z$ may not be unique. In fact the scheme $Z$
is uniquely determined by $P$ for an arbitrary $P\in \sigma _t(X)^{\dagger}$ only under very restrictive conditions (see e.g. Theorem \ref{e1} for a sufficient condition). However, we think that it useful
to see $\sigma _t(X)^{\dagger}$ as a union on the 
various strings $t_1\ge \cdots \ge t_s$, even when this is not a disjoint union.
\end{remark}

We recall the following definition (\cite{a}).

\begin{definition}\label{join} Fix now integral and non-degenerate subvarieties $X_1,\dots ,X_t\subset \mathbb{P}^r$ (repetitions are allowed). The join $J(X_1,\dots ,X_t)$ of $X_1, \dots ,X_t$
is the closure in $\mathbb{P}^r$ of the union of all $(t-1)$-dimensional vector spaces spanned by $t$ linearly independent points $P_1,\dots ,P_t$ with $P_i\in X_i$ for all
$i$.
\end{definition}

From Definition \ref{join} we obviously have that $\sigma _t(X_1) = J(\underbrace{X_1,\dots ,X_1}_{t})$.

\begin{definition}
Let $S(X_1,\dots, X_t) \subset X_1\times \cdots \times X_t\times \mathbb{P}^r$ be the closure
of the set of all $(P_1,P_2,\dots ,P_t,P)$ such that $P\in \langle \{P_1,\dots ,P_t\}\rangle$ and $P_i\in X_i$ for all $i$. We call $S(X_1,\dots, X_t)$ the abstract join of the subvarieties $X_1,\dots ,X_t$
of $\mathbb{P}^r$. 
\end{definition}

The abstract join $S(X_1,\dots, X_t)$ is an integral projective variety and we have $\dim (S(X_1,\dots ,X_t)) = t-1 +\sum _{i=1}^{t} \dim (X_i)$.
The projection of $X_1\times \cdots \times X_t\times \mathbb{P}^r \to \mathbb{P}^r$ induces a proper morphism $u_{X_1,\dots ,X_t}: S(X_1,\dots ,X_t) \to \mathbb{P}^r$ such
that $u_{X_1,\dots ,X_t}(S(X_1,\dots ,X_t)) = J(X_1,\dots ,X_t)$. The embedded join has the expected dimension $t-1 +\sum _{i=1}^{t} \dim (X_i)$ if and only
if $u_{X_1,\dots ,X_t}$ is generically finite.

\section{Curvilinear subschemes and tangential varieties to Veronese varieties}

From now on in this paper we fix  integers $m\ge 2$, $d\ge 3$ and take $N:= \binom{m+d}{m}-1$ and $X:= X_{m,d}$ the Veronese embedding of $\mathbb{P}^m$ into  $\mathbb{P}^N$.

\begin{definition}\label{tau}
Let
$\tau (X)\subseteq \mathbb{P}^N$ be the tangent developable of
$X$, i.e. the closure in $\mathbb{P}^N$ of the union of all embedded tangent spaces $T_PX$, $P\in X_{reg}$:
$$
\tau(X):=\overline{\bigcup_{P\in X} T_PX}
$$
\end{definition}

\begin{remark}
Obviously
$\tau (X) \subseteq \sigma _2(X)$ and $\tau (X)$ is integral. Since
$d\ge 3$, the variety $\tau (X)$ is a hypersurface of $\sigma _2(X)$.
\end{remark}

\begin{definition}\label{taut}
For each integer
$t\ge 3$ let
$\tau (X,t)
\subseteq
\mathbb{P}^N$ be the join of
$\tau (X)$ and
$\sigma _{t-2}(X)$:
$$
\tau (X,t):=J(\tau (X), \sigma _{t-2}(X)).
$$
\end{definition}

We recall that $\min \{n,t(m+1)-2\}$ is the expected dimension of $\tau (X,t)$.

Here we fix integers $d, t$ with $t\ge 2$, $d$ not too small and look at $\tau (X,t)$ from many points of view.

\begin{remark}
The set $\tau (X,t)$ is nothing else than the closure inside $\sigma _t(X)$ of the largest stratum of our stratification, i.e.  is the stratum given by $\mbox{Hilb}^t(X)_c[2,1,\cdots , 1]$ (Lemma \ref{zz1}).
\end{remark}

For any integral projective scheme $W$, any effective Cartier divisor $D$ of $W$ and any closed subscheme $Z$ of $W$ the residual scheme
$\mbox{Res}_D(Z)$ of $Z$ with respect to $D$ is the closed subscheme of $W$ with $\mathcal {I}_Z:\mathcal {I}_D$ as its ideal sheaf. For every $L\in \mbox{Pic}(W)$ we have
the exact sequence
\begin{equation}\label{eqb1}
0 \to \mathcal {I}_{\mbox{Res}_D(Z)}\otimes L(-D) \to \mathcal {I}_Z\otimes L \to \mathcal {I}_{Z\cap D,D}\otimes (L\vert D)\to 0
\end{equation}

The long cohomology exact sequence of (\ref{eqb1}) gives the following well-known result, often called the Castelnuovo's lemma.

\begin{lemma}\label{z1.0} 
Fix $L\in \mbox{Pic}(Y)$ for $Y\subset \mathbb{P}^r$ any integral projective variety. Then 
$$h^i(Y,\mathcal {I}_Z\otimes L) \le h^i(Y,\mathcal {I}_{\mbox{Res}_D(Z)}\otimes L(-D)) + h^i(D, \mathcal {I}_{Z\cap D,D}\otimes (L\vert D))$$
for every $i\in \mathbb {N}$.
\end{lemma}

\begin{notation}
For any $Q\in \mathbb {P}^m$ and any integer $k\ge 2$ let $kQ$ denote the  $(k-1)$-infinitesimal neighborhood of $Q$ in $\mathbb {P}^m$, i.e. the
closed subscheme of $\mathbb {P}^m$ with $(\mathcal {I}_Q)^k$ as its ideal sheaf. The scheme $kQ$ will be called
a $k$-point of $\mathbb {P}^m$. 
\end{notation}

We give here the definition of a $(2,3)$-point as it is in \cite{cgg}, p. 977.

\begin{definition}
 Fix a line
$L\subset
\mathbb {P}^m$ and a point 
$Q\in L$. The $(2,3)$ point of $\mathbb {P}^m$ associated to $(Q,L)$ is the closed subscheme $Z(Q,L)\subset \mathbb {P}^m$ with
$(\mathcal {I}_Q)^3+(\mathcal {I}_L)^2$ as its ideal sheaf. 
\end{definition}

In
\cite{bcgi}, Lemma 3.5, by using the theory of inverse systems, it is proved that the tangent space to the second osculating
variety to Veronese variety is dominated by
$4Q$, with
$Q\in X_{m,d}$, exactly as
$3Q$ dominates the tangent developable of $X_{m,d}$. Hence our computations with $4Q$ done in Lemma \ref{h2} may be useful
for joins of the second osculating variety of a Veronese and several copies of the Veronese.

Notice that $2Q \subset Z(Q,L) \subset 3Q$.

\begin{remark}\label{h0.0}
Let $Z = Z_1\sqcup Z(Q,L)$ be a closed subscheme of $\mathbb{P}^m$ for $Z_1\subset \mathbb{P}^m$ a $0$-dimensional scheme. Since $Z(Q,L) \subset 3Q$, if $h^1(\mathbb {P}^m,\mathcal {I}_{3Q\cup Z_1}(d))=0$, then $h^1(\mathbb {P}^m,\mathcal {I}_Z(d))=0$.
\end{remark}

\begin{lemma}\label{h0}
Fix an integer $t$ such that $(m+1)(t-2) + 2m < N$ with $N={m+d \choose d}-1$ and general $P_0,\dots ,P_{t-2}\in \mathbb {P}^m$
and a general line $L\subset \mathbb {P}^m$ such that $P_0\in L$. Set 
$$Z:= Z(P_0,L) \bigcup (\cup _{i=1}^{t-2} 2P_i), \ \ \ Z':= 3P_0 \bigcup (\cup _{i=1}^{t-2} 2P_i).$$ 

\quad (i) If $h^1(\mathbb {P}^m,\mathcal {I}_Z(d))=0$, then $\dim (\tau (X,t))
= t(m+1)-2$.

\quad (ii) If $h^1(\mathbb {P}^m,\mathcal {I}_{Z'}(d))=0$, then $\dim (\tau (X,t))
= t(m+1)-2$.
\end{lemma}

\begin{proof}
If $t=2$ then $\tau(X,t)=\tau(X)$ and the part (i) for this case is proved in \cite{cgg}. The case $t\ge 3$ of part (i) follows from the case $t=2$ and Terracini's lemma(\cite{a}, part (2) of Corollary 1.11), because $\tau (X,t)$
is the join of $\tau (X)$ and $t-2$ copies of $X$. Part (ii) follows from part (i) and Remark \ref{h0.0}.
\end{proof}

\begin{remark}\label{bn}
Let $A \subset \mathbb {P}^m$, $m\ge 2$, be a connected curvilinear subscheme of degree $3$. Up to a projective
transformation
there are two classes of such schemes: the collinear ones (i.e. $A$ is contained in a line, i.e. $\nu _d(A)$ is contained
in a degree $d$ rational normal curve) and the non-collinear ones, i.e. the ones that are contained in a smooth
conic of $\mathbb {P}^m$. We have $h^1(\mathbb {P}^m,\mathcal {I}_A(1)) >0$ if and only if $A$ is contained in a line.
Thus the semicontinuity theorem for cohomology gives that the set of all $A$'s  not contained in a line form a non-empty open subset of the corresponding stratum  $(3,0,\dots ,0)$ and, in this case, we will say that $A$ is {\it not collinear}. The family of all such schemes $A$ covers an integral variety
of dimension $3m-2$. If $d\ge 5$ any non-collinear one appears as the scheme computing
the border rank of the point of $\sigma _3(X)\setminus \sigma _2(X)$ with symmetric rank $2d-1$ (\cite{bgi}, Theorem 34).\end{remark}

\begin{lemma}\label{h1}
Fix integers $m \ge 2$ and $d\ge 5$. If $m\le 4$, then assume $d\ge 6$. Set $\alpha := \lfloor \binom{m+d-1}{m}/(m+1)\rfloor$.
Let $Z_i \subset \mathbb {P}^m$, $i=1,2$, be a general union of $i$ triple points and $\alpha -i$ double points.
Then $h^1(\mathcal {I}_{Z_i}(d))=0$.
\end{lemma}

\begin{proof}
Fix a hyperplane $H$ of $\mathbb {P}^m$ and call $E_i$ the union of $i$ triple points of $\mathbb {P}^m$ with support on $H$ with $i\in \{1,2\}$.
Hence $E_i\cap H$ is a disjoint union of $i$ triple points of $H$. Since $d\ge 5$, we have $h^1(H,\mathcal {I}_{H\cap
E_i}(d))=0$.   Let $W_i\subset \mathbb {P}^m$ be a general union of $\alpha -i$ double points for $i\in \{1,2\}$. Since $W_i$ is
general, we have
$W_i\cap H = \emptyset$. 
\\
If we prove that $h^1(\mathcal {I}_{E_i\cup W_i}(d))=0$, then, by semicontinuity, we get also that $h^1(\mathcal{I}_Z(d))=0$ for $i\in \{1,2\}$.
\\
By Lemma \ref{z1.0} it is sufficient to prove $h^1(\mathcal {I}_{\mbox{Res}_H(W_i\cup
E_i)}(d-1))=0$. 
\\
Since
$W_i\cap H=\emptyset$, we have $\mbox{Res}_H(W) = W$ and $\mbox{Res}_H(W_i\cup E_i) = W_i\sqcup \mbox{Res}_H(E_i)$. Hence $\mbox{Res}_H(W_i\cup E_i)$
is a general union of $\alpha$ double points, with the only restriction that the reductions of two of these double points
are contained in the hyperplane $H$. Any two points of $\mathbb {P}^m$, $m\ge 2$, are contained
in some hyperplane. The group $\mbox{Aut}(\mathbb {P}^m)$ acts transitively on the set of all hyperplanes
of $\mathbb {P}^m$. The cohomology
groups of projectively equivalent subschemes of
$\mathbb {P}^m$ have the same dimension. Hence we may consider $W_i\sqcup \mbox{Res}_H(E_i)$ as a general union of $\alpha$ double points
of $\mathbb {P}^m$. Since $(m+1)\alpha  \le \lfloor \binom{m+d-1}{m}/(m+1)\rfloor$, $d-1 \ge 4$ and $d-1 \ge 5$ if $m\le 4$, a
famous theorem
of Alexander and Hirschowitz on the dimensions of all secant varieties to Veronese varieties gives $h^1(\mathcal {I}_{\mbox{Res}_H(W_i\cup
E_i)}(d-1))=0$ (see \cite{ah1}, \cite{ah2}, \cite{ah0},\cite{c}, \cite{bo})
\end{proof}

\begin{lemma}\label{h2}
Fix integers $m \ge 2$ and $d\ge 6$. If $m\le 4$, then assume $d\ge 7$. Set $\beta := \lfloor \binom{m+d-2}{m}/(m+1)\rfloor$.
Let $Z \subset \mathbb {P}^m$ be a general union of one quadruple point and $\beta - 1$ double points.
Then $h^1(\mathcal {I}_Z(d))=0$.
\end{lemma}

\begin{proof}
Fix a hyperplane $H$ and call $E$ a quadruple point of $\mathbb {P}^m$ with support on $H$.
Hence $E\cap H$ is a quadruple point of $H$. Since $d\ge 2$, we have $h^1(H,\mathcal {I}_{H\cap
E}(d))=0$.   Let $W\subset \mathbb {P}^m$ be a general union of $\beta -1$ double points. Since $W$ is general,
we have $W\cap H = \emptyset$. 
\\
If we prove that $h^1(\mathcal {I}_{E\cup W}(d))=0$ then, by semicontinuity, we get also that $h^1(\mathcal{I}_{Z}(d))=0$.
By Lemma \ref{z1.0} it is sufficient to prove $h^1(\mathcal {I}_{\mbox{Res}_H(W\cup
E)}(d-1))=0$. 
\\
Since
$W\cap H=\emptyset$, we have $\mbox{Res}_H(W) = W$ and $\mbox{Res}_H(W\cup E) = W\sqcup \mbox{Res}_H(E)$. Hence $\mbox{Res}_H(W\cup E)$
is a general union of $\beta -1$ double points and one triple
point with support on $H$. Since $\mbox{Aut}(\mathbb {P}^m)$ acts transitively, the scheme $\mbox{Res}_H(W\cup E)$ may be seen as a general disjoint
union of $\beta -1$ double points and one triple point.
Now it is sufficient to apply the case $i=1$ of Lemma \ref{h1} for the integer $d':=d-1$.
\end{proof}

\begin{proposition}\label{j1}
Set $\alpha := \lfloor \binom{m+d-1}{m}/(m+1)\rfloor$. Fix an integer $t\ge 3$ such that $t\le \alpha -1$.
There is a non-empty and irreducible codimension $1$ algebraic subset $\Gamma _1$ of $\sigma _t(X)$ with the following
property. For every $P\in \Gamma _1$ there is a scheme $Z_P \subset X$ such that $P\in \langle Z_P\rangle$ and $Z_P$ has one
connected component of degree $2$ and $t-2$ connected components of degree $1$.
\end{proposition}

\begin{proof}
Lemma \ref{h1} and Terracini's lemma (\cite{a}, part (2) of Corollary 1.11) give that the join $\tau (X,t)$ (see Definition \ref{taut}) has the expected
dimension. This is equivalent to say that the set of all points $P\in \langle Z_1\cup \{P_1,\dots ,P_{t-2}\}\rangle$
with $Z_1$ a tangent vector of $X$ has the expected dimension, i.e. codimension $1$ in $\sigma _t(X)$.
Obviously $\tau (X,t)  \ne \emptyset$ and $\Gamma _1\ne \emptyset$. The set $\Gamma _1$ is irreducible,
because it is an open subset of a join of irreducible subvarieties. 
\end{proof}

The proof of Proposition \ref{j1} can be analogously repeated for the following two propositions. 

\begin{proposition}\label{j2}
Set $\alpha := \lfloor \binom{m+d-1}{m}/(m+1)\rfloor$. Fix an integer $t\ge 3$ such that $t\le \alpha -2$.
There is a non-empty and irreducible codimension $2$ algebraic subset $\Gamma _2$ of $\sigma _t(X)$ with the following
property. For every $P\in \Gamma _2$ there is a scheme $Z_P \subset X$ such that $P\in \langle Z_P\rangle$ and $Z_P$ has two
connected components of degree $2$ and $t-4$ connected components of degree $1$.
\end{proposition}

\begin{proof} This proposition can be proved in the same way of Proposition \ref{j1} just quoting the case $i=2$ of Lemma \ref{h1} instead of
the case $i=1$ of the same lemma.
\end{proof}

\begin{proposition}\label{j3}
Set $\beta := \lfloor \binom{m+d-2}{m}/(m+1)\rfloor$. Fix an integer $t\ge 3$ such that $t\le \beta -1$.
There is a non-empty and irreducible codimension $2$ algebraic subset $\Gamma _3$ of $\sigma _t(X)$ with the following
property. For every $P\in \Gamma _3$ there is a scheme $Z_P \subset X$ such that $P\in \langle Z_P\rangle$ and $Z_P$ has $t-3$ connected components of degree $1$
and one
connected component which is curvilinear, of degree $3$ and  non-collinear.
\end{proposition}

\begin{proof} This proposition can be proved in the same way of Proposition \ref{j1} just quoting Lemma \ref{h2} instead of Lemma \ref{h1} and using Remark \ref{bn}.
\end{proof}

Notice that we may take $\Gamma _1= \sigma _t(X)_c[2,1,\dots ,1]$, $\Gamma _2= \sigma _t(X)_c[2,2,1,\dots ,1]$ and as $\Gamma _3$ a non-empty open subset of $\sigma _t(X)_c[3,1,\dots ,1]$.

\begin{remark} Observe that if we interpret the Veronese variety $X_{m,d}$ as the variety that parameterizes the projective classes of homogeneous polynomials of degree $d$ in $m+1$ variables that can be written as $d$-th powers of linear forms then:
\begin{itemize}
\item The elements  $F\in\Gamma_1$ can all be written in the following two forms:
$$F=L^{d-1}M+L_1^d+ \cdots + L_{t-2}^d,$$
$$F=M_1^d+ \cdots + M_d^d+L_1^d+ \cdots + L_{t-2}^d.$$
\item The elements  $F\in\Gamma_2$ can all be written in the following two forms:
$$F=L^{d-1}M+L'^{d-1}M'+L_1^d+ \cdots + L_{t-4}^d;$$
$$F=M_1^d+ \cdots + M_d^d+M_1^{'d}+ \cdots + M_d^{'d}+L_1^d+ \cdots + L_{t-4}^d.$$
\item The elements  $F\in\Gamma_3$ can be written either in one of the two following forms:
$$F=L^{d-2}Q+L_1^d+ \cdots + L_{t-3}^d;$$
$$F=N_1^d+ \cdots + N_{2d-1}^d+L_1^d+ \cdots + L_{t-3}^d;$$
or in one of the two following forms:
$$F=L^{d-1}M+L_1^d+ \cdots + L_{t-3}^d,$$
$$F=M_1^d+ \cdots + M_d^d+L_1^d+ \cdots + L_{t-3}^d.$$
\end{itemize}
where $L,L'M,M'L_1, \ldots , L_{t-2}, M_1, \ldots , M_d,M'_1, \ldots , M'_d, N_1, \ldots , N_{2d-1}$ are all linear forms and $Q$ is a quadratic form. Actually  $M_1, \ldots , M_d$ and $M'_1, \ldots , M'_d$ are binary forms (see \cite{bgi}, Theorem 32 and Theorem 37).
\end{remark}

\section{The ranks and border ranks of points of $\Gamma _i$}\label{S3}

Here we compute the rank $r_X(P)$ for certain points $P\in \tau (X,t)$ when $t$ is not too big
with respect to $d$. The cases $t=2$ are contained in \cite{bgi}, Theorems 32 and 34. The case $t=4$ is contained in \cite{bb}, Theorem 1.

We first handle the border rank.

\begin{theorem}\label{e1}
Fix an integer $t$ such
that $2 \le t \le \lfloor (d-1)/2\rfloor$. For each $P\in \sigma_t(X)\setminus (\sigma_t^{0}(X) \cup \sigma_{t-1}(X))$ there
is a unique
$W_P\in \mbox{Hilb}^t(X)$ such that $P\in \langle W_P\rangle$.

\quad (a) The constructible set $\sigma_t(X)^{\dagger}$
is non-empty, irreducible and of dimension $(m+1)t-2$. For a general $P\in \sigma_t(X)^{\dagger}$ the associated $W\subset X$ computing $b_X(P)$ has a connected component of degree $2$ (i.e. a tangent
vector) and $t-2$ reduced connected components.

\quad (b)  We have a set-theoretic
partition $\sigma_t(X)^{\dagger } = \sqcup _{\underline{t}\in B(t)} \sigma (\underline{t})$, where $A(t)$ is defined in Notation \ref{n1}, in which each set $\sigma(\underline{t})$ is an irreducible and non-empty constructible subset of dimension $(m+1)t-1-t+l(\underline{t})$, where $l(\underline{t})$ is defined in Notation \ref{Az}. The strata $\sigma (2,1,\dots ,1)$ is the only open stratum and all the other strata are in the closure of $\sigma (2,1,\dots ,1)$.

\quad (c) $\sigma (2,2,\dots ,1)$ and $\sigma (3,1,\dots ,1)$ are the only strata of codimension $1$ of $\sigma_t(X)^{\dagger }$.

\quad (d) If $t_1\ge 3$ (resp. $t_2\ge 3$), then the stratum $\sigma (t_1,\dots ,t_s)$ is in the closure of $\sigma (3,1,\dots ,1)$ (resp. $\sigma (2,2,\dots ,1)$).

\quad (e) The complement of $\sigma_t(X)^{\dagger}$ inside $\sigma_t(X)\setminus (\sigma_t^{0}(X) \cup \sigma_{t-1}(X))$ has codimension
at least $3$ if $t\ge 3$, or it is empty if $t=2$.
\end{theorem}

\begin{proof}
Fix $P\in \sigma_t(X)\setminus \sigma _{t-1}(X)$. Remark \ref{ee1} gives $\beta (X) = d+1 \ge t$. Therefore Lemma \ref{a1} gives the existence
of some $W\subset X$ such that $\deg (W)=t$, $P\in \langle W\rangle$ and $W$ is smoothable. Since $2t \le d+1$, we can use \cite{bb}, Lemma 1 to say that $W$ is unique.
Moreover, if $A\subset X$ is a degree $t$ smoothable subscheme, $Q\in \langle A\rangle$ and
$Q\notin \langle A'\rangle $ for any $A'\subsetneq A$, then Lemma \ref{a1} gives $Q\in \sigma _t(X)\setminus \sigma _{t-1}(X)$.
If $A$ is curvilinear, then it is smoothable and $\cup _{A'\subsetneq A} \langle A'\rangle \subsetneq \langle A\rangle$.
Hence each degree $t$ curvilinear subscheme $W$ of $X$ contributes a non-empty open subset $U_W$
of the $(t-1)$-dimensional projective space $\langle W\rangle$ and $U_{W_1}\cap U_{W_2} = \emptyset$
for all curvilinear $W_1, W_2$ such that $W_1\ne W_2$.
Hence 
$$\sigma _t(X)^{\dagger} = \sqcup _{\underline{t}\in A(t)} (\sqcup _{W\in \mbox{Hilb}^t(X)[\underline{t}]} U_W).$$
Each algebraic set $B_{\underline{t}}:= \sqcup _{W\in \mbox{Hilb}^t(X)[\underline{t}]} U_W$
is irreducible and of dimension $t-1 + tm +l(\underline{t})-t$. This partition of $\sigma _t(X)^{\dagger}$ into
non-empty irreducible constructible subsets is the partition claimed in part (b).

Parts (b), (c) and (d) follows from Lemma \ref{zz1}.

Now we prove part (e). Every element of $\mbox{Hilb}^2(X)$ is either
a tangent vector or the disjoint union of two points. Hence $\mbox{Hilb}^2(X) =\mbox{Hilb}^2(X)_c$.
Hence we may assume $t\ge 3$. Fix $P\in \sigma
_t(X)\setminus (\sigma _t^{0}(X)
\cup
\sigma _{t-1}(X))$ such that
$P\notin
\sigma _t(X)^{\dagger} $. By Lemma \ref{a1} there is a smoothable $W\subset X$ such that $\deg (W) =t $ and $P\in \langle
W\rangle$. Since $2t \le \beta (X)$, such a scheme is unique. Hence it is sufficient to prove that the set $\mathbb {B}_t$ of all
$0$-dimensional smoothable schemes with degree $t$ and not curvilinear have dimension at most $mt -3$. 
\\
Call
$\mathbb {B}_t(s)$ the set of all $W\in \mathbb {B}_t$ with exactly $s$ connected components. 
\\
First we assume that $W$ is
connected. Set $\{Q\}:= W_{red}$. Since in the local Hilbert scheme of $\mathcal {O}_{X,Q}$ the smoothable colength $t$ ideals
are parametrized by an integral variety of dimension $(m-1)(t-1)$ and a dense open subset of it is formed by the ideals associated
to a curvilinear subschemes, we have $\dim (\mathbb {B}_t(1)) \le m+(m-1)(t-1)-1 = mt -t = \dim (\mbox{Hilb}^t(X)_c)-t$.
\\
Now we assume $s\ge 2$. Let $W_1,\dots ,W_s$ be the connected
components of $W$, with at least one of them, say $W_s$, not curvilinear.
Set $t_i = \deg (W_i)$. We have $t_1+\cdots +t_s = t$. Since $W_s$ is not curvilinear, we have $t_s\ge 3$
and hence $t-s \ge 2$. Each $W_i$ is
smoothable. Hence each
$W_i$,
$i<s$, depends on at most
$m+(m-1)(t_i-1)= mt_i +1-t_i$ parameters. We saw that $\mathbb {B}_{t_s}(1)$ depends on
at most $mt_s-t_s$ parameters. Hence $\dim (\mathbb {B}_t(s)) \le mt + s-1 -t$.
\end{proof}

\begin{proposition}\label{f1}
Assume $m \ge 2$. Fix integers $d, t$ such that $2 \le t \le d$. Fix a curvilinear scheme $A\subset \mathbb {P}^m$ such
that $\deg (A) =t$ and $\deg (A\cap L) \le 2$ for every line $L\subset \mathbb {P}^m$. Set
$Z:= \nu _d(A)$.
Fix $P\in \langle Z\rangle$ such that $P\notin \langle Z'\rangle$ for any $Z'\subsetneq Z$.
Then $b_X(P) = t$ and $Z$ is the only $0$-dimensional scheme
$W$ such that $\deg (W) \le t$ and $P\in \langle W\rangle$.
\end{proposition}

\begin{proof}
Since $t\le d+1$, $Z$ is linearly independent. Since $Z$ is curvilinear, Lemma \ref{a2} gives the existence
of many points $P'\in \langle Z\rangle$ such that $P'\notin \langle Z'\rangle$ for any $Z'\subsetneq Z$.
Let $W\subset X$
be a minimal degree subscheme such that
$P\in
\langle W\rangle$. Set
$w:= \deg (W)$. The minimality of $w$ gives $w \le t$. If $w = t$,
then we assume $W\ne Z$. Now it is sufficient
to show that these conditions give a contradiction. Write $Z:= \nu _d(A)$ and $W = \nu _d(B)$
with $A$ and $B$ subschemes of $\mathbb {P}^m$, $\deg (A)=t$ and $\deg (B)=w$. We have $P\in \langle W\rangle \cap \langle
Z\rangle$, then, since
$W\ne Z$, by \cite{bb}, Lemma 1, the scheme $W\cup Z$ is linearly dependent. We have $\deg (B\cup A) \le t+w \le
2d$. Since
$W\cup Z$ is linearly dependent, we have $h^1(\mathcal {I}_{B\cup A}(d))>0$. Hence, by \cite{bgi}, Lemma 34,
there is a line $R\subset \mathbb {P}^m$ such that $\deg (R\cap (B\cup A)) \ge d+2$. By assumption
we have $\deg (R\cap A) \le 2$. Hence $\deg (B\cap R)\ge d$. In our set-up we get $w=d$ and $B\subset R$. Since $P\in \langle W\rangle$,
we get $P\in \langle \nu _d(R)\rangle$.
That means that $P$ belongs to the linear span of a rational normal curve. Therefore the border rank of $P$ is computed by a curvilinear scheme which has length $\le \lfloor (d+1)/2\rfloor$, a contradiction. 
\end{proof}

\begin{proposition}\label{e2+}
Fix a line $L\subset \mathbb {P}^m$ and set $D:= \nu _d(L)$. Fix positive integers
$t_1, s_1$, a $0$-dimensional scheme
$Z_1\subset D$ such that $\deg (Z_1)=t_1$ and $S_1\subset X\setminus D$
such that $\sharp (S_1)=s_1$. Assume $2 \le t_1 \le d/2$, $0 \le s_1 \le d/2$, that $Z_1$ is not reduced and $\dim (\langle
D\cup S_1\rangle )=d+s_1$. Fix $P\in \langle Z_1\cup S_1\rangle$ such that $P\notin \langle W\rangle$ for any $W\subsetneq
Z_1\cup S_1$. We have $\sharp (\langle Z_1\rangle \cap \langle \{P\}\cup S_1\rangle )=1$. Set $\{Q\}:= \langle Z_1\rangle \cap \langle \{P\}\cup S_1\rangle$.
Then
$b_X(P)=t_1+s_1$, $r_X(P) = d+2+s_1-t_1$, $Z_1\cup S_1$ is the only subscheme of $X$ computing $b_X(P)$ and
every subset of $X$ computing $r_X(P)$ contains $S_1$. If $2s_1<d$, then every subset of $X$ computing
$r_X(P)$ is of the form $A\cup S_1$ with $A\subset D$, $\sharp (A) =d+2-s_1$ and $A$ computing $r_D(Q)$.
\end{proposition}

\begin{proof}
Obviously $b_X(P)\le t_1+s_1$. Since $P\in \langle Z_1\cup S_1\rangle \subset \langle D\cup S_1\rangle$, $P\notin
\langle S_1\rangle$ and $\langle D\rangle $ has
codimension $s_1$ in $\langle D\cup S_1\rangle$, the linear subspace $ \langle Z_1\rangle \cap \langle \{P\}\cup
S_1\rangle$ is non-empty and $0$-dimensional, $\{Q\}$. Since $\deg (Z_1) \le d+1 = \beta (X) =\beta (D)$ (Remark \ref{ee1}), the scheme
$Z_1$ is linearly independent. Since $P\notin \langle W\rangle$ for any $W\subsetneq
Z_1\cup S_1$, we have $\langle Z_1\rangle \cap \langle \{P\} \cup S_1\rangle \ne \emptyset$. Since
$\langle Z_1\rangle \subset \langle D\rangle$, we get $\{Q\}= \langle Z_1\rangle \cap \langle \{P\} \cup S_1\rangle$.
Hence $Z_1$ compute $b_D(Q)$ (Lemma \ref{a1}). By Lemma \ref{a1} we also have $b_X(Q)=b_D(Q) =t_1$.
Since $Z_1$ is not reduced, we have $r_D(Q) = d+2-t_1$ (\cite{cs} or \cite{lt}, theorem 4.1, or \cite{bgi}, \S 3).
We have $r_X(Q)=r_D(Q)$ (\cite{ls}, Proposition 3.1, or \cite{lt}, subsection 3.2).
Write 
$Z_1 =
\nu _d(A_1)$ and 
$S_1 =
\nu _d(B_1)$ with $A_1, B_1\subset \mathbb {P}^m$. Lemma \ref{a1} gives $b_X(P) \le t_1+s_1$. Assume $b_X(P) \le t_1+s_1-1$ and take $W =\nu _d(E)$ computing $b_X(Q)$ for certain $0$-dimensional scheme $E\subset \mathbb{P}^m$. Hence $\deg (W) \le 2t_1+2s_1-1$. Since $P\in \langle W\rangle \cap \langle Z_1\cup S_1\rangle$, by the already quoted \cite{bb}, Lemma 1, we get
$h^1(\mathbb {P}^m,\mathcal {I}_{E\cup A_1\cup B_1}(d)) >0$. Hence there is a line $R\subset \mathbb {P}^m$ such that $\deg (R\cap (E\cup Z_1\cup S_1)) \ge d+2$.

First assume $R = L$. Hence $L\cap (A_1\cup B_1) = A_1$. Hence
$\deg (E\cap L) \ge d+2-t_1$. Set $E':= E\cap L$, $E'':= E\setminus E'$, $W':= \nu _d(E')$ and $W'':= \nu _d(E'')$. Since
$P\in \langle W'\cup W''\rangle$, there is $O\in \langle W'\rangle$ such that
$P\in \langle \{O\}\cup W''\rangle$. Hence $b_X(P) \le b_X(O) +\deg (W'')$. Since $O\in \langle D\rangle$, we have $r_X(O)\le r_D(O) \le \lfloor (d+2)/2\rfloor < d+2-t_1 \le \deg (W')$, contradicting the assumption
that $W$ computes $b_X(P)$.

Now assume $R\ne L$. Since the scheme
$L\cap R$ has degree $1$, while the scheme $A_1\cap L$ has degree $t_1$, we get $\deg (R\cap E) \ge d+2-s_1 > (d+2)/2$. As above we get a contradiction.

Now assume $b_X(P) = t_1+s_1$, but that $W\ne Z_1\cup S_1$ computes $b_X(P)$. As above we get a line $R$
such that $\deg (W\cup Z_1\cup S_1) \ge d+2$ and  this line $R$ must be $L$.
Since $P\in \langle Z_1\cup S_1\rangle$, there is $U\in \langle D\rangle$ such that $Z_1$ computes the border $D$-rank of $U$ and $P \in \langle U\cup S_1\rangle$. Take
$A\subset D$ computing $r_D(U)$. By \cite{cs} or \cite{lt}, Theorem 4.1, or \cite{bgi} we have $\sharp (A) = d+2-t_1$. Since $P\in \langle A\cup S_1\rangle$ and $A\cap S_1=\emptyset$, we
have $r_X(P) \le d+2+s_1-t_1$. Assume the existence of some $S\subset X$ computing $r_X(P)$ and such that $\sharp (S) \le
d+1+s_1-t_1$. Hence $\deg (S\cup S_1\cup Z_1) \le d+1+2s_1 \le 2d+1$.  Write $S = \nu _d(B)$. We proved that $Z_1\cup S_1$ computes $b_X(P)$. By \cite{bb}, Theorem 1,
we have $B = B_1\sqcup S_1$ with $B_1 = L\cap B$. Hence $\sharp (B_1) \le d+1-t_1$. Since $P\in \langle B_1\cup S_1\rangle$, there is $V\in \langle B_1\rangle$ such that
$P\in \langle V\cup S_1\rangle$. Hence $r_X(P) \le r_X(V) +s_1$. Since $B$ computes $r_X(P)$ and $V\in \langle B_1\rangle$, we get $r_X(V) = \sharp (B_1)$ and
that $B_1$ computes $r_X(V)$. Since $\nu _d(B_1) \subset D$, we have $V = Q$. Recall that $b_X(Q) = b_D(Q)$ and that $Z_1$ is the only subscheme of $X$ computing
$r_X(Q)$. We have $r_X(Q) =r_D(Q) = d+2-t_1$. Hence $\sharp (B_1)\ge d+2-t_1$, a contradiction.

If $2s_1<d$, then the same proof works even if $\sharp (B) =d+2+s_1-t_1$
and prove that any set computing $r_X(P)$ contains $S_1$.
\end{proof}

\begin{lemma}\label{c2==} 
Fix a hyperplane $M \subset \mathbb {P}^m$ and $0$-dimensional schemes $A, B$ such that $B$ is reduced, $A\ne B$,
$h^1(\mathcal {I}_A(d)) =h^1(\mathcal {I}_B(d))=0$ and $h^1(\mathbb {P}^m,\mathcal {I}_{\mbox{Res}_M(A\cup B)}(d-1)) =0$. Set
$Z:= \nu _d(A)$,  $S:= \nu _d(B)$. Then
$h^1(\mathbb{P}^m, \mathcal {I}_{A\cup B}(d))= h^1(M,\mathcal {I}_{(A\cup B)\cap M}(d))$ and $Z$ and $S$ are linearly
independent. Assume the
existence $P\in \langle Z\rangle \cap \langle S\rangle$ such that $P\notin \langle Z'\rangle$ for any $Z'\subsetneq Z$
and $P\notin \langle S'\rangle$ for any $S'\subsetneq S$. Set $F:= (B\setminus (B\cap M))\cap A$.
Then $B = (B\cap M)\sqcup F$ and $A = (A\cap M)\sqcup F$.
\end{lemma}

\begin{proof}
Since  $h^1(\mathcal {I}_A(d)) =h^1(\mathcal {I}_B(d))=0$, both $Z$ and $S$ are linearly independent.
Since $h^2(\mathcal {I}_{A\cup B}(d-1)) =0$, the residual sequence
$$0 \to \mathcal {I}_{\mbox{Res}_M(A\cup B)}(d-1) \to \mathcal {I}_{A\cup B}(d) \to \mathcal {I}_{(A\cup B)\cap M}(d) \to 0.$$
gives $h^1(\mathbb{P}^m, \mathcal {I}_{A\cup B}(d))= h^1(M,\mathcal {I}_{(A\cup B)\cap M}(d))$. Assume
the existence of $P$ as in the statement. Set $B_1:= (B\cap M)\cup F$.

\quad (a) Here we prove that $B = (B\cap M)\cup F$, i.e. $B =B_1$. Since $P\notin \langle S'\rangle$ for any $S'\subsetneq S$,
it is sufficient to prove $P\in \langle \nu _d(B_1)\rangle$. Since $Z$ and $S$ are linearly independent,
Grassmann's formula gives $\dim (\langle Z\rangle \cap\langle S\rangle )  = \deg (Z\cap S)-1 + h^1(\mathbb{P}^m, \mathcal {I}_{A\cup B}(d))$. Since
$\mbox{Res}_M(A\cup B_1) \subseteq \mbox{Res}_M(A\cup B)$ and $h^1(\mathbb {P}^m,\mathcal {I}_{\mbox{Res}_M(A\cup B)}(d-1)) =0$, we have $h^1(\mathbb{P}^m, \mathcal {I}_{A\cup B_1}(d))= h^1(M,\mathcal {I}_{(A\cup B_1)\cap M}(d))$. Since $M\cap (A\cup B_1) = M\cap (A\cup B)$, we get
 $h^1(\mathbb{P}^m, \mathcal {I}_{A\cup B_1}(d))=h^1(\mathbb{P}^m, \mathcal {I}_{A\cup B}(d))$. Since both schemes $Z$ and $\nu _d(B)$ are
linearly independent, Grassmann's formula gives $\dim (\langle Z\rangle \cap\langle \nu _d(B)\rangle )  = \deg (A\cap B)-1 + h^1(\mathbb{P}^m, \mathcal {I}_{A\cup B}(d))$. Since both schemes $Z$ and $\nu _d(B_1)$ are
linearly independent, Grassmann's formula gives $\dim (\langle Z\rangle \cap\langle \nu _d(B_1)\rangle )  = \deg (A\cap B_1)-1 + h^1(\mathbb{P}^m, \mathcal {I}_{A\cup B}(d))$.
Since $A\cap B_1 = A\cap B$, we get  $\dim (\langle Z\rangle \cap\langle S\rangle ) = \dim (\langle Z\rangle \cap \langle \nu _d(B_1)\rangle $. Since
$\langle Z\rangle \cap \langle \nu _d(B_1)\rangle \subseteq \langle Z\rangle \cap \langle S\rangle$, we get $\langle Z\rangle \cap \langle \nu _d(B_1)\rangle = \langle Z\rangle \cap \langle S \rangle$.
Hence $P\in \langle \nu _d(B_1)\rangle$.

\quad (b) In a very similar way we get $A = (A\cap M)\sqcup F$ (see steps (b), (c) and (d) of the proof of Theorem 1
in \cite{bb}).\end{proof}

\begin{theorem}\label{f2}
Assume $m \ge 3$. Fix integers $d \ge 5$ and $3 \le t \le d$. Fix 
a degree $2$ connected subscheme $A_1\subset L$ and a reduced set $A_2\subset \mathbb {P}^m\setminus L$, such
that $\sharp (A_2) =t-2$ and $h^1(\mathbb {P}^m,\mathcal {I}_A(d))=0$, for $A:= A_1\cup A_2$. Set $Z_i:= \nu _d(A_i)$, $i=1,2$, and
$Z:= Z_1\cup Z_2$. Assume that $A$ is in linearly general position in $\mathbb {P}^m$. Fix $P\in \langle Z\rangle$ such that $P\notin \langle Z'\rangle$ for any $Z'\subsetneq Z$. Then $b_X(P)=t$
and $r_X(P) = d+t-2$.
\end{theorem}

\begin{proof}
Since $h^1(\mathbb {P}^m,\mathcal {I}_A(d))=0$, then the scheme $Z$ is linearly independent. Proposition \ref{f1} gives $b_X(P)=t$. Fix a set $B \subset \mathbb {P}^m$
such that $S:= \nu _d(B)$ computes $r_X(P)$. Assume $r_X(P)<d+t-2$, i.e. $\sharp (S) \le d+t-3$. Since $t \le d$, we have $r_X(P) + t \le 3d-3$.

\quad (a) Until step (g) we assume $m=3$.
We have $h^1(\mathbb {P}^m,\mathcal {I}_{A\cup B}(d)) >0$ (\cite{bb}, Lemma 1). Hence $A\cup B$ is not in linearly general
position (see \cite{eh}, Theorem 3.2). Hence there is a plane $M\subset \mathbb {P}^3$ such
that $\deg (M\cap (A\cup B))\ge 4$. Among all such planes we take one, say $M_1$, such that the integer $x_1:= \deg (M_1\cap (A\cup B))$ is maximal. Set
$E_1:= A\cup B$ and $E_2:= \mbox{Res}_{M_1}(E_1)$. Notice that $\deg (E_2)=\deg (E_1)-x_1$. Define
inductively the planes $M_i \subset \mathbb {P}^3$, $i \ge 2$, the schemes $E_{i+1}$, $i\ge 2$, and the integers $x_i$,
$i\ge 2$, by the condition that
$M_i$ is one of the planes such that the integer $x_i:= \deg (M_i\cap E_i)$ is maximal and then set $E_{i+1}:= \mbox{Res}_{M_i}(E_i)$. We
have
$E_{i+1}\subseteq E_i$ (with strict inclusion if $E_i \ne \emptyset$) for all $i\ge 1$ and $E_i=\emptyset$ for all $i \gg 0$.
For all integers
$t$ and
$i
\ge 1$ there is the residual exact sequence
\begin{equation}\label{eqd1}
0 \to \mathcal {I}_{E_{i+1}}(t-1) \to \mathcal {I}_{E_i}(t) \to \mathcal {I}_{E_i\cap M_i,M_i}(t) \to 0.
\end{equation}
Let $u$ be the minimal positive integer $i$ such that and $h^1(M_i,\mathcal {I}_{M_i\cap E_i}(d+1-i))>0$.
Use at most $r_X(P) +t$ times the exact sequences (\ref{eqd1}) to prove the existence of such an integer $u$. Any degree $3$
subscheme of $\mathbb {P}^3$ is contained in a plane. Hence for any $i\ge 1$ either $x_i \ge 3$ or $x_{i+1} = 0$.
Hence $x_i \ge 3$ for all $i \le u-1$. Since $r_X(P) + t \le 3d$, we get $u \le d$.

\quad (b) Here we assume $u=1$. Since
$A$ is in linearly general position, we have
$\deg (M_1\cap A)
\le 3$. First assume $x_1\ge 2d+2$. Hence $\sharp (B) \ge \sharp (B\cap M_1)\ge 2d-1>d+t-3$, a contradiction. Hence $x_1\le 2d+1$.
Since $h^1(M_1,\mathcal {I}_{M_1\cap E_1}(d)) > 0$, there is a line $T\subset M_1$ such that $\deg (T\cap E_1)\ge d+2$  (\cite{bgi}, Lemma 34).
Since $A$ is in linearly general position, we have $\deg (A\cap T) \le 2$. Hence $\deg (T\cap B)\ge d$. Assume for the moment $h^1(\mathbb {P}^3,\mathcal
{I}_{E_2}(d-1)) >0$. Hence
$x_2
\ge d+1$. Since by hypothesis $d \ge 4$, $x_2\le x_1$ and $x_1+x_2 \le 3d+1$, we have $x_2\le 2d-1$. Hence \cite{bgi}, Lemma 34,
applied to the integer
$d-1$ gives the existence of a line $R\subset \mathbb {P}^3$ such that $\deg (E_2\cap R)\ge d+1$. Since
$A$ is in linearly general position, we also get $\deg (R\cap E_2) \le 2$
and hence $\deg (R\cap B\cap E_2) \ge d-1$. Hence $\sharp (S) \ge 2d-1$, a contradiction. Now assume $h^1(\mathbb {P}^3,\mathcal
{I}_{E_2}(d-1)) =0$. Lemma \ref{c2==} gives the existence of a set $F\subset \mathbb {P}^3\setminus M_1$ such that
$A = (A\cap M_1)\sqcup F$ and $B = (B\cap M_1)\sqcup F$. Hence $\sharp (F) = \deg (A) -\deg (A\cap M_1) \ge t-1$. Since
$\sharp (B\cap M_1) \ge d$, we obtained a contradiction.

\quad (c) Here and in steps (d), (e), and (f) we assume $m=3$ and $u\ge 2$. 
We first look at the possibilities for the integer $u$. Since every degree $3$ closed subscheme of $\mathbb {P}^3$ is contained in a plane, either $x_i \ge 3$ or $x_{i+1}=0$. Since $r_X(P) +t \le 3d-3$, we get
$x_i=0$ for all $i>d$. Hence $u \le d$. We have
$x_u \ge d+3-u$ (e.g. by \cite{bgi}, Lemma 34). Since the sequence $x_i$, $i\ge 1$, is non-increasing, we get
$r_X(P) +2+t-2 \le u(d+3-u)$. Since the function
$s\mapsto s(d+3-s)$ is concave in the interval $[2,d+1]$, we get $u\in \{2,3,d\}$. 

\quad (d) Here we assume $u=2$. Since $3d+1 \ge x_1+x_2 \ge 2x_2$, we get
$x_2 \le 2(d-1)+1$. Hence there is a line $R\subset \mathbb {P}^3$ such that $\deg (E_2\cap R) \ge d+1$. We claim that $x_1 \ge d+1$. Indeed, since $A\cup B \nsubseteq R$, there
is a plane $M\subset R$ such that $\deg (M\cap (A\cup B)) > \deg ((A\cup B)\cap R) \ge d+1$. The maximality property of $x_1$ gives $x_1 \ge d+2$.
Since $A$ is in linearly general position, we have $\deg (A\cap R) \le 2$ and $\deg (A\cap M_1) \le 3$. Hence $\deg (B\cap E_2\cap R) \ge d-1$
and $r_X(P) \ge (x_1-3)+d-1 \ge 2d-2 \ge d+t-2$, a contradiction.

\quad (e) Here we assume $u=3$. Since $h^1(M_3,\mathcal {I}_{M_3\cap E_3}(d-2))>0$, there is a line $R\subset M_3$ such that $\deg (E_3\cap T) \ge d$. This is absurd, because $x_1\ge x_2\ge x_3 \ge d$
and $x_1+x_2+x_3 \le r_X(P) + t \le d+2t -3 \le 3d-3$.

\quad (f) Here we assume $u=d$. The condition
``~$h^1(\mathcal {I}_{M_d\cap E_d}(1)) >0$~'' says that either $M_d\cap E_d$ contains a scheme of length $\ge 3$ contained in a line
$R$ or $x_d\ge 4$. Since $x_d \ge 3$, we have $r_X(P) +t \ge x_1+\cdots +x_d \ge 3d$. Since $t\le d$ and $r_X(P) \le d+t-3$, this is absurd.

\quad (g) Here we assume $m>3$. We make a similar proof, taking as $M_i$, $i\ge 1$, hyperplanes of $\mathbb {P}^m$. Any $0$-dimensional scheme of degree at most $m$
of $\mathbb {P}^m$ is contained in hyperplane. Hence either
$x_i \ge m$ or $x_{i+1}=0$. With these modification we repeat the proof of the case $m=3$.
\end{proof}

The following example is the transposition of \cite{bb1}, Example 2, to our set-up.

\begin{example}\label{dd1}
Fix a smooth plane conic $C\subset \mathbb {P}^m$, $m\ge 2$,
and positive integers $d\ge 5$, $x, y$, $a_i$, $1\le i \le x$, and $b_j$, $1\le j \le y$, such
that $\sum _{i=1}^{x} a_i + \sum _{j=1}^{y} b_j =2d+2$. Fix $x+y$ distinct points $P_1, \dots ,P_x, Q_1, \dots ,Q_y$
of $C$. Let $A\subset C$ be the effective degree $\sum _{i=1}^{x} a_i$ divisor of $C$ in which each $P_i$ appear with
multiplicity $a_i$. Let $B\subset C$ be the effective degree $\sum _{j=1}^{j} b_j$ divisor of $C$ in which each $Q_j$ appear
with multiplicity $b_j$. Since $C$ is projectively normal, $h^0(C,\mathcal {O}_C(d)) = 2d+1$
and $h^1(C,\mathcal {I}_E(d))=0$ for every divisor $E$ of $C$ with degree at most $2d+1$, the set $\langle \nu _d(A)\rangle
\cap \langle \nu _d(B)\rangle$ is a unique point, $P$, $P\notin \langle \nu _d(A')\rangle$
for any $A'\subsetneq A$ and $P\notin \langle \nu _d(B')\rangle$ for any $B'\subsetneq B$. Since $h^1(C,\mathcal {I}_E(d))=0$ for every divisor $E$ of $C$ with degree at most $2d+1$,
it is easy to check that $b_X(P) = \min \{\deg (A),\deg (B)\}$. Thus $P$ is contained in two different quasi-strata
of $\sigma _t(X_{m,d})^{\dagger}$ for $t\ge \max \{\deg (A), \deg (B)\}$. If $\deg (A) = \deg (B) =d+1$, then $P\in \sigma _{d+1}(X_{m,d})^{\dagger}\setminus \sigma _d(X_{m,d})$
and both $A$ and $B$ compute the border rank of $P$.
\end{example}

\providecommand{\bysame}{\leavevmode\hbox to3em{\hrulefill}\thinspace}

\end{document}